\documentclass[english,11pt,reqno]{amsart}

\usepackage{xcolor}
\usepackage{mathptmx}
\usepackage{amsthm}
\usepackage{aliascnt}
\usepackage{amsmath,amsfonts,amssymb}
\usepackage{mathtools}
\usepackage{geometry}
\usepackage{hyperref}

\hypersetup{
  colorlinks,
  linkcolor={red!50!black},
  citecolor={blue!62!black},
  urlcolor={blue!80!black},
  pdftitle={Some Examples and Counterexamples in Oka Theory},
  pdfauthor={Yun-Heng Du},
  pdfsubject={Examples and counterexamples concerning Oka properties}
}
\theoremstyle{plain}

\newaliascnt{theorem}{thm}
\newtheorem{theorem}[theorem]{Theorem}
\aliascntresetthe{theorem}
\newaliascnt{corollary}{thm}
\newtheorem{corollary}[corollary]{Corollary}
\aliascntresetthe{corollary}
\newaliascnt{lemma}{thm}
\newtheorem{lemma}[lemma]{Lemma}
\aliascntresetthe{lemma}
\newaliascnt{proposition}{thm}
\newtheorem{proposition}[proposition]{Proposition}
\aliascntresetthe{proposition}
\theoremstyle{definition}
\newaliascnt{definition}{thm}
\newtheorem{definition}[definition]{Definition}
\aliascntresetthe{definition}
\numberwithin{equation}{section}

\usepackage[hyperpageref]{backref}

\usepackage[nameinlink,noabbrev]{cleveref}
\crefname{thm}{Theorem}{Theorems}
\Crefname{thm}{Theorem}{Theorems}
\crefname{theorem}{Theorem}{Theorems}
\Crefname{theorem}{Theorem}{Theorems}
\crefname{lemma}{Lemma}{Lemmas}
\Crefname{lemma}{Lemma}{Lemmas}
\crefname{proposition}{Proposition}{Propositions}
\Crefname{proposition}{Proposition}{Propositions}
\crefname{corollary}{Corollary}{Corollaries}
\Crefname{corollary}{Corollary}{Corollaries}
\crefname{definition}{Definition}{Definitions}
\Crefname{definition}{Definition}{Definitions}
\crefname{section}{Section}{Sections}
\Crefname{section}{Section}{Sections}

\allowdisplaybreaks[4]
\newcommand{\A}{\mathbb A}
\newcommand{\B}{\mathbb B}
\newcommand{\C}{\mathbb C}
\newcommand{\D}{\mathbb D}

\newcommand{\Pj}{\mathbb P}
\newcommand{\Q}{\mathbb Q}
\newcommand{\R}{\mathbb R}
\newcommand{\Z}{\mathbb Z}
\newcommand{\Cstar}{\C^*}
\newcommand{\norm}[1]{\left\lVert #1\right\rVert}
\newcommand{\abs}[1]{\left\lvert #1\right\rvert}
\DeclareMathOperator{\Bl}{Bl}
\DeclareMathOperator{\Crit}{Crit}
\DeclareMathOperator{\Ima}{Im}
\DeclareMathOperator{\Jac}{Jac}
\DeclareMathOperator{\id}{id}
\DeclareMathOperator{\rank}{rank}

\title{Some Examples and Counterexamples in Oka Theory}
\author{Yun--Heng Du}
\date{}

\subjclass[2020]{Primary 32Q56; Secondary 32E20, 32E30, 32H02,
14F45, 14R10, 32S45, 55N10}

\keywords{Oka manifold, polynomially convex set, Hartogs triangle,
blow-up, algebraic basic Oka property, weight filtration}

\begin{document}

\begin{abstract}
This paper gives examples and counterexamples in Oka theory: two about deleting closed
sets from complex Euclidean space, one about blowing up along a connected Oka
center, and one about deforming continuous maps to regular maps. Two positive
results show that $\C^3\setminus S$ is Oka for every closed set
$S\subset\R^3$, and that the complement of the closed Hartogs triangle
in $\C^2$ is Oka. In contrast, for every $n\ge3$ there is a proper holomorphic embedding
$\C\hookrightarrow\C^n$ whose image $A$ is a closed connected smooth curve
biholomorphic to $\C$, but whose blow-up $\Bl_A\C^n$ is Brody volume
hyperbolic and hence not Oka. Finally, for every $n\ge2$, there exist a smooth
connected affine algebraic variety $X$
and a continuous map $X\to\C^n\setminus\{0\}$ that is not homotopic to any
regular map $X\to\C^n\setminus\{0\}$; equivalently, $\C^n\setminus\{0\}$ fails
the algebraic basic Oka property (aBOP).
\end{abstract}

\maketitle

\section{Introduction}

Let $Y$ be a complex manifold. A map $f\colon K\to Y$ from a compact set
$K\subset\C^m$ is called holomorphic if it extends holomorphically to a
neighborhood of $K$. The Oka property will be formulated through the convex
approximation property.

\begin{definition}\cite{Forstneric2006Runge}
The manifold $Y$ is \emph{Oka} if,
for every integer $m\ge1$, every nonempty
compact convex set $K\subset\C^m$, and every map $f\colon K\to Y$
holomorphic on $K$, there are entire maps $F_j\colon\C^m\to Y$ converging
uniformly to $f$ on $K$.
\end{definition}

For a systematic account of Oka manifolds, see
\cite[Chapter~5]{ForstnericBook2011}.
Uniform convergence on $K$ is independent of the choice of a distance
inducing the topology of $Y$, and the Oka property is invariant under
biholomorphisms.

This paper gives examples and counterexamples in Oka theory. Two
concern complements of specified closed
sets in complex Euclidean space; the third concerns the standard analytic
blow-up of a complex manifold along a connected closed Oka submanifold; the
fourth concerns replacing a continuous map
from a smooth affine algebraic variety to a space of the form
$\C^n\setminus\{0\}$, where $n\ge2$, by a regular map in its homotopy class.

\par\vspace{0.5\baselineskip}
\noindent\textbf{Closed subsets of $\R^3$.}
The closed-set form of Kusakabe's complement theorem used below is formulated
in terms of polynomial convexity, so we first fix the relevant convention.

\begin{definition}\label{def:polynomial-convexity}
For a compact set $K\subset\C^n$, its polynomial hull is
$$
  \widehat K
  =\left\{z\in\C^n:
    \abs{P(z)}\le\sup_{x\in K}\abs{P(x)}
    \text{ for every }P\in\C[z_1,\ldots,z_n]\right\}.
$$
By convention, $\widehat\varnothing=\varnothing$, and $K$ is called
\emph{polynomially convex} when $\widehat K=K$. For a closed set
$E\subset\C^n$, a
\emph{compact exhaustion} is an increasing sequence $\{E_j\}_{j\ge1}$ of
compact subsets whose union is $E$ and such that every compact subset of
$E$ is contained in some $E_j$. In this setting, set
$$
  \widehat E=\bigcup_{j\ge1}\widehat{E_j},
$$
and call $E$ polynomially convex when $\widehat E=E$.
\end{definition}

This is the convention in Kusakabe's theorem used below. It is independent of the
exhaustion: if $\{E_j\}$
and $\{F_j\}$ are two compact exhaustions, then every $E_j$ is contained in
some $F_k$, whence $\widehat{E_j}\subset\widehat{F_k}$, and the reverse
inclusion follows after interchanging the two exhaustions.

Kusakabe proved that certain complements of closed polynomially convex sets
are Oka and applied this theorem to closed subsets of the standard real
subspaces $\R^k\subset\C^n$. His result
\cite[Theorem~1.6 and Corollary~1.7]{Kusakabe2024} covered all pairs $(n,k)$ except
$$
  (n,k)=(2,1),\qquad(2,2),\qquad(3,3).
$$
Forstneri\v c and
Forn\ae ss Wold settled $(2,1)$
\cite[Proposition~4.9]{ForstnericWold2024} and recorded that $(2,2)$ and $(3,3)$
remained open \cite[Remark~4.10]{ForstnericWold2024}. We settle the case
$(3,3)$.

\begin{theorem}\label{thm:r3-complement}
If $S\subset\R^3$ is closed in the relative topology, then
$\C^3\setminus S$ is an Oka manifold.
\end{theorem}

The missing case $(3,3)$ is not caused by a failure of polynomial convexity:
every compact subset of $\R^m\subset\C^m$ is polynomially convex
\cite[Example~1]{Forstneric1994RungeComplements}. The difficulty is to meet the
geometric hypothesis in Kusakabe's theorem. In the present dimension, that
hypothesis requires a holomorphic change of coordinates and a splitting
$\C^3=\C\times\C^2$ in which the Euclidean norm of the second coordinate is
bounded on the closed set by a constant times one plus the modulus of the
first coordinate. A quadratic polynomial automorphism records the squared
length of the last two real coordinates in the imaginary part of the first
coordinate and thereby gives this estimate. The same automorphism works for
every relatively closed $S\subset\R^3$, so the theorem follows from
\cite[Theorem~1.6]{Kusakabe2024} without a new complement criterion. The
remaining exceptional pair $(2,2)$ is not addressed here.

\par\vspace{0.5\baselineskip}
\noindent\textbf{The closed Hartogs triangle.}
Write $\Cstar=\C\setminus\{0\}$.
Forstneri\v c asked in \cite[Problem~4.24]{Forstneric2023} whether the
complement of the closed Hartogs triangle
$$
  H=\{(z_1,z_2)\in\C^2:\abs{z_1}\le\abs{z_2}\le1\}
$$
is Oka. This question was again listed as open in the subsequent survey
\cite[p.~664]{Forstneric2026ICM}. An affirmative answer to Problem~4.24 follows.

\begin{theorem}\label{thm:hartogs-complement}
The complement $\C^2\setminus H$ of the closed Hartogs triangle is an Oka
manifold.
\end{theorem}

Forstneri\v c and Forn\ae ss Wold proved that
$(\C\times\Cstar)\setminus K$ is Oka whenever $K\subset\C^2$ is compact and
polynomially convex \cite[Corollary~3.2]{ForstnericWold2024}. The difficulty is
to relate the Hartogs complement to this model. After a complex line through
the origin is removed, the ratio of two linear coordinates records the
complex direction of a point. For two suitable lines, these directional
coordinates identify the corresponding portions of $\C^2\setminus H$ with
complements of the required form. The two portions are Zariski open and cover
the whole complement. Kusakabe's localization theorem states
that a connected complex manifold is Oka if every point has a Zariski-open Oka
neighborhood \cite[Theorem~1.4]{Kusakabe2021}, so it proves
\cref{thm:hartogs-complement}. Thus the Hartogs inequalities are handled by
two explicit local models.

\par\vspace{0.5\baselineskip}
\noindent\textbf{Blow-ups along connected Oka centers.}
The standard analytic blow-up of a complex manifold $X$ along a closed
complex submanifold $A$ replaces each point of $A$ by the projective space of
complex lines in its normal space. Denote the blow-up by $\Bl_A X$, its
blow-down map by $\pi\colon\Bl_A X\to X$, and its exceptional divisor by
$\pi^{-1}(A)$. Forstneri\v c asked which operations,
including blow-ups and blow-downs, preserve the Oka property in the list of
open problems from his 2011 Krems lecture
\cite[p.~143]{Forstneric2011Krems}. On the positive side, L\'arusson and
Truong proved that, for $n\ge2$, the blow-up of $\C^n$ along any algebraic
submanifold (not necessarily connected) is Oka
\cite[Theorem~1]{LarussonTruong2017}. On the
negative side, Kusakabe used discrete sets constructed by Rosay
and Rudin to obtain non-Oka blow-ups of affine spaces along suitable closed
discrete centers \cite[Lemmas~A.1--A.2 and Example~A.3]{Kusakabe2021}. Our
result allows the center to be connected, smooth, and Oka.

For an open set $U\subset\C^n$ and a holomorphic map
$F=(F_1,\ldots,F_n)\colon U\to\C^n$, set
$$
  \Jac F=\det\left(\frac{\partial F_i}{\partial z_j}\right)_{i,j=1}^n,
  \qquad
  \Crit(F)=\{z\in U:\Jac F(z)=0\}.
$$
The map is \emph{nondegenerate} when $\Jac F\not\equiv0$, and
$z\in U$ is a \emph{regular point} of $F$ when
$z\notin\Crit(F)$.
More generally, if $Y$ is an $n$-dimensional complex manifold, a holomorphic
map $f\colon\C^n\to Y$ is nondegenerate when $df$ has rank $n$ at some point.

\begin{definition}
An
$n$-dimensional complex manifold $Y$ is \emph{Brody volume hyperbolic} if
there is no nondegenerate holomorphic map $\C^n\to Y$.
\end{definition}

Here a continuous map $p\colon X\to Z$ is \emph{proper} if $p^{-1}(K)$ is
compact for every compact set $K\subset Z$.

\begin{theorem}\label{thm:connected-blowup}
For every integer $n\ge3$, there is a proper holomorphic embedding
$\gamma\colon\C\hookrightarrow\C^n$ whose image $A=\gamma(\C)$ is a closed connected
smooth complex curve such that $\Bl_A\C^n$ is Brody volume hyperbolic and
hence is not Oka.
\end{theorem}

The discrete-center examples of Kusakabe use closed discrete sets constructed
by Rosay and Rudin that every nondegenerate entire map
$\C^n\to\C^n$ must meet. A connected center creates an additional difficulty.
After such a set $D$ is placed on a complex curve, the mere condition
$F^{-1}(D)\ne\varnothing$ gives no contradiction in the blow-up argument: one
needs an intersection outside $\Crit(F)$. We therefore refine the construction
in \cite[Lemmas~4.3--4.4 and Theorem~4.5]{RosayRudin1988} to obtain a closed discrete set
$D\subset\C^n$ such that every nondegenerate entire map
$F\colon\C^n\to\C^n$ satisfies
\begin{equation}\label{eq:regular-unavoidable-intro}
  F^{-1}(D)\not\subset\Crit(F).
\end{equation}
In the present construction the finite subsets on spheres are chosen
so that a forced intersection has a regular preimage. Taking
their union over successively larger spheres retains this stronger conclusion
and produces a closed discrete set.

An interpolation theorem of Forstneri\v c, Globevnik, and Rosay
\cite[Proposition~2]{ForstnericGlobevnikRosay1996} places $D$ on a properly
embedded copy $A=\gamma(\C)\subset\C^n$. A nondegenerate map to the blow-up would,
after composition with the blow-down map, contradict
\eqref{eq:regular-unavoidable-intro}, because the differential of the blow-down
map loses rank along the exceptional divisor. Compared with the earlier
discrete-center counterexamples, the center here is connected, smooth, and
biholomorphic to the Oka manifold $\C$. It is noncompact, so the theorem does
not settle blow-ups at a point or along a compact center.

\par\vspace{0.5\baselineskip}
\noindent\textbf{The algebraic basic Oka property.}
Here an algebraic manifold is a smooth connected complex algebraic variety,
whereas an affine algebraic variety used as a source need not be smooth or
connected. A regular map is a morphism of algebraic varieties.

\begin{definition}\cite[pp.~200--201]{LarussonTruong2019}
An algebraic manifold $Y$ has the \emph{algebraic basic Oka property}
(aBOP) if, for every affine algebraic variety $X$, every continuous map
$f\colon X\to Y$ is homotopic, in the Euclidean topology of complex points,
to a regular map $g\colon X\to Y$.
\end{definition}

The formulation in \cite[Definition~6.14(a)]{Forstneric2023} restricts the
source to an affine algebraic manifold, hence to a smooth connected source.
The source constructed below is smooth and connected, so
\cref{thm:punctured-abop} disproves aBOP under either source convention.

L\'arusson and Truong proved that a positive-dimensional
algebraic manifold fails aBOP if it is compact or admits a nonconstant regular
map from $\Pj^1$ \cite[Theorem~2]{LarussonTruong2019}. Neither hypothesis holds
for $\C^n\setminus\{0\}$: this variety is noncompact, and, for every regular map
$f\colon\Pj^1\to\C^n\setminus\{0\}$, each coordinate of its composite with the
inclusion $\C^n\setminus\{0\}\hookrightarrow\C^n$ is a global regular function
on $\Pj^1$ and hence is constant; therefore $f$ is constant. Forstneri\v c singled out
$\C^2\setminus\{0\}$ in \cite[Problem~6.17]{Forstneric2023}; his later survey continued
to single out $\C^2\setminus\{0\}$ as a basic test case for aBOP
\cite[p.~668]{Forstneric2026ICM}.

\begin{theorem}\label{thm:punctured-abop}
For every integer $n\ge2$, the punctured affine space
$\C^n\setminus\{0\}$ does not satisfy aBOP.
\end{theorem}

The one-dimensional target $\Cstar$ was already known not to satisfy aBOP
\cite[Theorem~4.2.3]{Dye2022}; the case $n=2$ of
\cref{thm:punctured-abop} answers Problem~6.17.

The construction following \cite[Problem~6.17]{Forstneric2023} already gives a
regular map from a smooth affine quadric that represents a generator of
$\pi_3(\C^2\setminus\{0\})$. Thus nontrivial target homotopy alone cannot
obstruct aBOP; the source must also retain algebraic information. We use the
explicit smooth affine variety
$$
  X_n=\{(x,y)\in\C^2:y^2=x^3-x\}\times(\Cstar)^{2n-2}.
$$
Throughout the paper, $H^\bullet(-;\Q)$ denotes singular cohomology with rational
coefficients. The topology of $X_n$ gives a continuous map
$F_n\colon X_n\to\C^n\setminus\{0\}$ for which
\begin{equation}\label{eq:nonzero-pullback-intro}
  F_n^*\colon
  H^{2n-1}(\C^n\setminus\{0\};\Q)
  \longrightarrow H^{2n-1}(X_n;\Q)
\end{equation}
is nonzero. Deligne associates to the rational cohomology of every complex
algebraic variety an increasing filtration by subspaces, called the weight
filtration, and regular pullbacks preserve this filtration
\cite[Theorem~3.2.5(iii)]{Deligne1971}. Set
$Y_n=\C^n\setminus\{0\}$. The calculation in \cref{ab:prop:weights} shows that
this preservation forces every regular map $X_n\to Y_n$ to have zero pullback
in degree $2n-1$. This is incompatible with
\eqref{eq:nonzero-pullback-intro}, since homotopic maps induce the same map in
cohomology. This obstruction applies to $\C^n\setminus\{0\}$ for every $n\ge2$.

\par\vspace{0.5\baselineskip}
\noindent\textbf{Structure of the paper.}
In \Cref{sec:r3-main}, the polynomial convexity required at the point
of use is verified and the quadratic polynomial automorphism is constructed;
the section then applies Kusakabe's theorem to closed subsets of $\R^3$.
\Cref{sec:hartogs} contains an explicit coordinate description of the Hartogs
complement, the compactness and polynomial convexity arguments, and the
application of Kusakabe's localization theorem.
\Cref{sec:blowup-main} is devoted to the closed discrete set that every
nondegenerate entire map meets at a regular point, its embedding in a
connected Oka center, and the proof of the blow-up theorem. \Cref{sec:abop-main}
constructs the source varieties and continuous maps and proves the
obstruction to aBOP coming from the weight filtration.

\section{Complements of closed subsets of
\texorpdfstring{$\R^3$}{R3}}\label{sec:r3-main}

The starting point is the polynomial automorphism used to verify the
hypothesis of Kusakabe's theorem.
Identify $\C^3$ with $\C\times\C^2$ by using coordinates
$(\zeta,w_1,w_2)$ on the target.

\begin{proposition}\label{r3:prop:tube-estimate}
The polynomial map
$$
  \Phi\colon\C^3\longrightarrow\C\times\C^2,
  \qquad
  \Phi(z_1,z_2,z_3)
  =\bigl(z_1+i(z_2^2+z_3^2),z_2,z_3\bigr),
$$
is a polynomial automorphism with inverse
$$
  \Psi(\zeta,w_1,w_2)
  =\bigl(\zeta-i(w_1^2+w_2^2),w_1,w_2\bigr).
$$
Moreover, for every $x=(x_1,x_2,x_3)\in\R^3$, if
$$
  \Phi(x)=(\zeta,w)\in\C\times\C^2,
$$
then
$$
  \norm{w}\leq1+\abs{\zeta}.
$$
Consequently,
$$
  \Phi(\R^3)
  \subset
  \left\{(\zeta,w)\in\C\times\C^2:
    \norm{w}\leq1+\abs{\zeta}
  \right\}.
$$
\end{proposition}

\begin{proof}
Direct substitution gives $\Psi\circ\Phi=\id_{\C^3}$ and
$\Phi\circ\Psi=\id_{\C\times\C^2}$, so $\Phi$ is a polynomial automorphism
with inverse $\Psi$.

Set
$$
  r=\sqrt{x_2^2+x_3^2}\geq0.
$$
The definition of $\Phi$ gives
$$
  \zeta=x_1+i(x_2^2+x_3^2)=x_1+ir^2,
  \qquad
  w=(x_2,x_3).
$$
Thus $\norm w=r$ and $r^2=\Ima\zeta\leq\abs\zeta$. Since
$r\leq1+r^2$ for $r\geq0$, it follows that
$$
  \norm{w}=r\leq1+r^2\leq1+\abs{\zeta}.
$$
\end{proof}

The following is the required form of Kusakabe's theorem for a possibly
unbounded closed set \cite[Theorem~1.6]{Kusakabe2024}.

\begin{theorem}\label{thm:tube-criterion}
Suppose that $n>1$ and that $E\subset\C^n$ is a closed polynomially convex set
for which there exist a constant $C>0$ and a holomorphic automorphism $\varphi$ of
$\C^n$ such that, under the splitting $\C^n=\C^{n-2}\times\C^2$,
$$
  \varphi(E)\subset
  \{(z,w)\in\C^{n-2}\times\C^2:
       \norm w\le C(1+\norm z)\}.
$$
Then $\C^n\setminus E$ is Oka.
\end{theorem}

\begin{proof}[Proof of Theorem~\ref{thm:r3-complement}]
Let $S\subset\R^3$ be closed in the relative topology. Since the standard
real subspace $\R^3$ is closed in $\C^3$, the set $S$ is closed in $\C^3$.

For each $j\geq1$, set
$$
  S_j=S\cap\{z\in\C^3:\norm z\leq j\}.
$$
Each $S_j$ is a compact subset of $\R^3$ and is therefore polynomially
convex by \cite[Example~1]{Forstneric1994RungeComplements}. The sequence
$\{S_j\}$ is a compact exhaustion of $S$, so
\cref{def:polynomial-convexity} gives
$$
  \widehat S
  =\bigcup_{j=1}^{\infty}\widehat{S_j}
  =\bigcup_{j=1}^{\infty}S_j
  =S.
$$
Thus $S$ is a closed polynomially convex subset of $\C^3$.

The automorphism and estimate in \cref{r3:prop:tube-estimate} give
$$
  \Phi(S)
  \subset\Phi(\R^3)
  \subset
  \left\{(\zeta,w)\in\C\times\C^2:
    \norm{w}\leq1+\abs{\zeta}
  \right\}.
$$
This is the hypothesis of Theorem~\ref{thm:tube-criterion} for the
splitting $\C^3=\C\times\C^2$ and $C=1$. Hence $\C^3\setminus S$ is Oka.
\end{proof}

\section{The closed Hartogs triangle}
\label{sec:hartogs}

Set $H=\{(z_1,z_2)\in\C^2:|z_1|\leq|z_2|\leq1\}$ and
$\Omega=\C^2\setminus H$.
Here $\D=\{z\in\C:|z|<1\}$ denotes the unit disc.
The set $H$ is compact. Its complement is connected because
$$
  \Omega
  =\{\abs{z_1}>\abs{z_2}\}
   \cup\{\abs{z_2}>1\};
$$
the first set is biholomorphic to $\Cstar\times\D$, the second is
$\C\times(\C\setminus\overline\D)$, and their intersection is nonempty.
Therefore $\Omega$ is a domain.

For each $a\in\D$, set
$L_a=\{(z_1,z_2)\in\C^2:z_2=az_1\}$ and $U_a=\Omega\setminus L_a$.
The line $L_a$ is closed in $\C^2$, so $U_a$ is open in $\Omega$ and hence is
a complex manifold. On $\C^2\setminus L_a$, consider the map
$$
  \Psi_a\colon\C^2\setminus L_a\longrightarrow\C\times\Cstar,
  \qquad \Psi_a(z_1,z_2)
  =\left(\frac{z_1}{z_2-az_1},z_2-az_1\right)=(u,v).
$$
Here $v=z_2-az_1$ vanishes precisely on $L_a$. Under the simultaneous scaling
$(z_1,z_2)\mapsto(\lambda z_1,\lambda z_2)$ with $\lambda\in\Cstar$, both $z_1$
and $v$ are multiplied by $\lambda$, so $u=z_1/v$ is unchanged and therefore
depends only on the projective direction $[z_1:z_2]$. On the chart
$\Pj^1\setminus\{[1:a]\}$, consisting exactly of the directions not represented
by $L_a$, it is an affine coordinate with inverse $u\mapsto[u:1+au]$.
Thus $\Psi_a$ records the projective direction in the coordinate $u$ and the
nonzero scalar in the coordinate $v$. The next lemma determines its image on
$U_a$.

\begin{lemma}\label{ht:lem:chart-model}
For every $a\in\D$, the set
$$
  K_a=\{(u,v)\in\C^2:|u|\leq|1+au|,\ |(1+au)v|\leq1\}
$$
is compact and polynomially convex, and $\Psi_a$ restricts to a
biholomorphism from $U_a$ onto $(\C\times\Cstar)\setminus K_a$.
\end{lemma}

\begin{proof}
Fix $a\in\D$. First determine the image of $H\setminus L_a$ under $\Psi_a$.
On $\C^2\setminus L_a$ one has $z_2-az_1\ne0$, so the functions
$$
  u=\frac{z_1}{z_2-az_1},
  \qquad
  v=z_2-az_1
$$
are well defined and holomorphic. They satisfy
$$
  z_1=uv,
  \qquad
  z_2=v+a z_1=(1+a u)v.
$$
Set
$$
  \Theta_a\colon\C\times\Cstar\longrightarrow\C^2\setminus L_a,
  \qquad
  \Theta_a(u,v)=(uv,(1+a u)v).
$$
For every $(u,v)\in\C\times\Cstar$, its image satisfies
$$
  (1+a u)v-a(uv)=v\neq0,
$$
so $\Theta_a(u,v)\notin L_a$. Direct substitution into the formulas for
$\Psi_a$ and $\Theta_a$ gives
$$
  \Psi_a\circ\Theta_a=\id_{\C\times\Cstar},
  \qquad
  \Theta_a\circ\Psi_a=\id_{\C^2\setminus L_a}.
$$
Thus $\Psi_a$ is a biholomorphism.

Let $(u,v)\in\C\times\Cstar$, and set
$(z_1,z_2)=\Theta_a(u,v)$. Since $v\neq0$, the two defining
inequalities for membership in $H$ become
$$
  |z_1|\leq|z_2|
  \quad\Longleftrightarrow\quad
  |uv|\leq |(1+a u)v|
  \quad\Longleftrightarrow\quad
  |u|\leq|1+a u|,
  \qquad
  |z_2|\leq1
  \quad\Longleftrightarrow\quad
  |(1+a u)v|\leq1.
$$
Moreover, $L_a\cap H=\{0\}$, so $H\setminus L_a=H\setminus\{0\}$.
The equivalences above therefore show that
\begin{equation}\label{ht:eq:image-forbidden}
  \Psi_a(H\setminus\{0\})
  =K_a\cap(\C\times\Cstar).
\end{equation}
Combining the definition of $U_a$ with \eqref{ht:eq:image-forbidden} gives
\begin{align}
  \Psi_a(U_a)
  &=\Psi_a\bigl((\C^2\setminus L_a)\setminus(H\setminus\{0\})\bigr)
    \notag\\
  &=(\C\times\Cstar)
      \setminus\bigl(K_a\cap(\C\times\Cstar)\bigr)
    \notag\\
  &=(\C\times\Cstar)\setminus K_a.
  \label{ht:eq:image-Ua}
\end{align}
The last equality is an equality of subsets of $\C\times\Cstar$: points
of $K_a$ with $v=0$ do not belong to that ambient manifold and hence
do not affect the set difference.

It remains to prove that $K_a$ is compact and polynomially convex. The
identity
$$
  |(1-|a|^2)u-\overline a|^2-1
  =(1-|a|^2)\bigl(|u|^2-|1+au|^2\bigr).
$$
Since $1-|a|^2>0$, this gives the equivalence
$$
  |u|\leq|1+a u|
  \quad\Longleftrightarrow\quad
  |(1-|a|^2)u-\overline a|\leq1.
$$
Consequently,
\begin{equation}\label{ht:eq:Ka-polynomial}
  K_a
  =\{(u,v)\in\C^2:
       |(1-|a|^2)u-\overline a|\leq1,
       \ |(1+a u)v|\leq1\}.
\end{equation}

For $(u,v)\in K_a$, the triangle inequality gives directly
\begin{equation}\label{ht:eq:u-bound}
  |u|\leq|1+au|\leq1+|a|\,|u|,
  \qquad\text{hence}\qquad
  |u|\leq\frac1{1-|a|}.
\end{equation}
Moreover, the inequality $|u|\leq|1+au|$ gives
$$
  1=|(1+au)-au|
  \leq|1+au|+|a|\,|u|
  \leq(1+|a|)|1+au|.
$$
It follows that $|1+au|\geq1/(1+|a|)$, and the inequality
$|(1+au)v|\leq1$ gives
$|v|\leq1+|a|$. Together with \eqref{ht:eq:u-bound}, this shows that $K_a$ is
bounded and, being closed, compact.

For the fixed parameter $a$, set
$$
  p_a(u,v)=(1-|a|^2)u-\overline a,
  \qquad
  q_a(u,v)=(1+a u)v.
$$
Then $p_a,q_a\in\C[u,v]$, and \eqref{ht:eq:Ka-polynomial} becomes
$$
  K_a=\{(u,v)\in\C^2:|p_a(u,v)|\leq1,
                               \ |q_a(u,v)|\leq1\}.
$$
Accordingly, $K_a$ is a compact polynomial polyhedron and is therefore
polynomially convex
by \cite[Lemma~2.7.4]{Hormander1990}.
Together with \eqref{ht:eq:image-Ua}, this proves the lemma.
\end{proof}

\begin{corollary}\label{ht:lem:chart-oka}
For every $a\in\D$, the complex manifold $U_a$ is Oka.
\end{corollary}

\begin{proof}
By \cref{ht:lem:chart-model}, the manifold $U_a$ is biholomorphic to
$(\C\times\Cstar)\setminus K_a$ for a compact polynomially convex set
$K_a\subset\C^2$. Forstneri\v c and Forn\ae ss Wold proved that if $K$ is a
compact polynomially convex set in $\C^N$ for some $N>1$, then
$(\C^{N-1}\times\Cstar)\setminus K$ is an Oka manifold
\cite[Corollary~3.2]{ForstnericWold2024}. Applying this
theorem with $N=2$ and $K=K_a$ therefore shows that $U_a$ is Oka.
\end{proof}

\begin{proof}[Proof of \cref{thm:hartogs-complement}]
By \cref{ht:lem:chart-oka}, the two subsets
$$
  U_0=\Omega\setminus L_0,
  \qquad
  U_{1/2}=\Omega\setminus L_{1/2}
$$
are Zariski-open Oka manifolds, and $\Omega=U_0\cup U_{1/2}$.
Then every point of $\Omega$ has a Zariski-open Oka neighborhood. Kusakabe's
localization theorem \cite[Theorem~1.4]{Kusakabe2021} states that if
$Y$ is a connected complex manifold such that every point of $Y$ has a
Zariski-open Oka neighborhood, then $Y$ is an Oka manifold. Applying this theorem with
$Y=\Omega$ therefore shows that $\Omega$ is Oka, which proves
\cref{thm:hartogs-complement}.
\end{proof}

\section{Blow-ups along connected Oka centers}
\label{sec:blowup-main}

The proof of \cref{thm:connected-blowup} rests on a closed discrete
set $D\subset\C^n$ such that, for every nondegenerate entire map
$F\colon\C^n\to\C^n$, there is a point $z\in\C^n$ with $F(z)\in D$ and
$\Jac F(z)\ne0$.
Rosay and Rudin obtained unavoidable discrete sets by choosing finite subsets
of expanding spheres
\cite[Lemmas~4.3--4.4 and Theorem~4.5]{RosayRudin1988}. The argument below
chooses these finite sets so that the required intersection occurs at a
regular point.

\begin{theorem}\label{bl:thm:regular-unavoidable}
For every integer $n>1$, there is a nonempty closed discrete set
$D\subset\C^n$ such that
$$
F^{-1}(D)\not\subset\Crit(F)
$$
for every nondegenerate entire map $F\colon\C^n\to\C^n$.
\end{theorem}

\begin{proof}[Proof of \cref{bl:thm:regular-unavoidable}]
Fix $n>1$, and write $\B^n$ for the unit ball in $\C^n$. Choose numbers
$0<a_1<a_2$, $0<r_1<r_2$, and $c>0$. Let
$\mathcal F$ be the family of
holomorphic maps $G\colon a_2\B^n\to r_2\B^n$ such that
$$
\norm{G(0)}<\frac{r_1}{2},
\qquad
\abs{\Jac G}\ge c\quad\text{at some point of }a_1\B^n.
$$
Choose nonempty finite sets
$E_1\subset E_2\subset\cdots\subset\partial(r_1\B^n)$ whose union is dense in
$\partial(r_1\B^n)$. We claim that, for some $\ell$, every $G\in\mathcal F$
satisfies
$$
G(a_1\B^n)\cap\partial(r_1\B^n)\ne\varnothing
\quad\Longrightarrow\quad
G^{-1}(E_\ell)\not\subset\Crit(G).
$$
If not, for every $\ell$ there are $F_\ell\in\mathcal F$ and
$z_\ell\in a_1\B^n$ such that
$$
F_\ell(z_\ell)\in\partial(r_1\B^n),
\qquad
F_\ell^{-1}(E_\ell)\subset\Crit(F_\ell).
$$
Choose $\xi_\ell\in a_1\B^n$ with $\abs{\Jac F_\ell(\xi_\ell)}\ge c$. The sets
$\overline{a_1\B^n}$ and $\overline{r_2\B^n}$ are compact, and the maps $F_\ell$
are uniformly bounded by $r_2$. By successively passing to subsequences,
compactness and Montel's theorem give three subsequences, selected along the
same original indices; the corresponding finite sets $E_\ell$ are retained as
well. Relabeling these subsequences and finite sets by $\ell$, we have
$$
z_\ell\longrightarrow w\in\overline{a_1\B^n},
\qquad
\xi_\ell\longrightarrow\xi\in\overline{a_1\B^n},
\qquad
F_\ell\longrightarrow F\quad\text{uniformly on compact subsets of }a_2\B^n,
$$
where $F\colon a_2\B^n\to\C^n$ is holomorphic and
$F(a_2\B^n)\subset\overline{r_2\B^n}$. Since
$\overline{a_1\B^n}\subset a_2\B^n$, the convergence is uniform on
$\overline{a_1\B^n}$, and therefore
$$
\norm{F_\ell(z_\ell)-F(w)}
\leq\norm{F_\ell(z_\ell)-F(z_\ell)}
   +\norm{F(z_\ell)-F(w)}\longrightarrow0.
$$
Cauchy estimates give uniform convergence of the first derivatives, and hence
of the Jacobian determinants, on $\overline{a_1\B^n}$. Thus
$$
\abs{\Jac F_\ell(\xi_\ell)-\Jac F(\xi)}
\leq\abs{\Jac F_\ell(\xi_\ell)-\Jac F(\xi_\ell)}
   +\abs{\Jac F(\xi_\ell)-\Jac F(\xi)}\longrightarrow0.
$$
It follows that
$$
F(w)\in\partial(r_1\B^n),
\qquad
\abs{\Jac F(\xi)}\ge c.
$$
In particular, $F$ is nondegenerate and $\norm{F(0)}\le r_1/2$.

Set $Q=\{z\in a_2\B^n:\Jac F(z)\ne0\}$. Since $\Jac F\not\equiv0$, its zero
set is a proper analytic subset of $a_2\B^n$. Its complement $Q$ is connected
and dense by \cite[Chapter~II, Remark~4.2]{DemaillyCADG}. Density and
$\norm{F(0)}\le r_1/2$ give
$F(Q)\cap r_1\B^n\ne\varnothing$. If
$F(Q)\subset\overline{r_1\B^n}$, density gives
$F(a_2\B^n)\subset\overline{r_1\B^n}$. Let
$e=F(w)/r_1$ and $g(z)=\langle F(z),e\rangle$, where the standard Hermitian
inner product is taken to be linear in the first variable. Since $\norm{e}=1$, the
Cauchy--Schwarz inequality gives
$$
\abs{g(z)}\leq\norm{F(z)}\leq r_1,
\qquad
g(w)=r_1.
$$
The maximum modulus principle gives $g\equiv r_1$. Hence, for every
$z\in a_2\B^n$,
$$
r_1=\abs{\langle F(z),e\rangle}\leq\norm{F(z)}\leq r_1.
$$
Equality in Cauchy--Schwarz and $g(z)=r_1$ now give
$F(z)=r_1e=F(w)$. Thus $F$ is constant, contradicting
$\abs{\Jac F(\xi)}\ge c$. Therefore $F(Q)$ also meets
$\C^n\setminus\overline{r_1\B^n}$, and its connectedness gives
$$
F(Q)\cap\partial(r_1\B^n)\ne\varnothing.
$$
Choose $p\in Q$ and $q=F(p)\in\partial(r_1\B^n)$. There are neighborhoods
$U$ of $p$ and $V$ of $q$ such that $F\colon U\to V$ is biholomorphic. Choose
$\delta>0$ with $\overline{B(q,\delta)}\subset V$, write
$\phi=(F|_U)^{-1}\colon V\to U$, and consider the holomorphic maps
$$
h_\ell=F_\ell\circ\phi\colon V\longrightarrow\C^n.
$$
The set $K=\phi(\overline{B(q,\delta)})$ is a compact subset of $U$. The
convergence $F_\ell\to F$ and the Cauchy estimates give uniform convergence
of both $F_\ell$ and $dF_\ell$ on $K$. Since $d\phi$ is bounded on
$\overline{B(q,\delta)}$, the identities
$$
h_\ell(y)-y=(F_\ell-F)(\phi(y)),
\qquad
dh_\ell(y)-I=(dF_\ell-dF)_{\phi(y)}\circ d\phi_y
$$
show that $h_\ell\to\operatorname{id}$ in the $\mathcal C^1$ topology on
$\overline{B(q,\delta)}$. For all sufficiently large $\ell$,
$$
\sup_{\overline{B(q,\delta)}}\norm{h_\ell-\operatorname{id}}<\frac\delta4,
\qquad
\sup_{\overline{B(q,\delta)}}\norm{dh_\ell-I}<\frac12.
$$
Here $\norm{h_\ell(y)-y}$ is the Euclidean norm on $\C^n$,
$\norm{dh_\ell(y)-I}$ is the induced operator norm, and $I$ is the identity map
on $\C^n$. If $dh_\ell(y)v=0$, then
$$
\norm v=\norm{(I-dh_\ell(y))v}
\leq\norm{I-dh_\ell(y)}\,\norm v<\frac12\norm v,
$$
so $v=0$. Thus $dh_\ell(y)$ is injective and hence, as an endomorphism of the
finite-dimensional vector space $\C^n$, invertible throughout
$\overline{B(q,\delta)}$.
For $y_0\in\overline{B(q,\delta/2)}$, the map
$T_{\ell,y_0}(y)=y+y_0-h_\ell(y)$ satisfies
$$
\norm{T_{\ell,y_0}(y)-q}<\frac{3\delta}{4},
\qquad
\norm{dT_{\ell,y_0}(y)}<\frac12
\quad\bigl(y\in\overline{B(q,\delta)}\bigr).
$$
The first estimate gives
$T_{\ell,y_0}(\overline{B(q,\delta)})\subset B(q,3\delta/4)
\subset\overline{B(q,\delta)}$, while the second and the convexity of the ball give
$\norm{T_{\ell,y_0}(y)-T_{\ell,y_0}(y')}\leq\frac12\norm{y-y'}$. The Banach
fixed-point theorem therefore gives $y\in\overline{B(q,\delta)}$ with
$T_{\ell,y_0}(y)=y$, or equivalently $h_\ell(y)=y_0$.

The reindexed sets $E_\ell$ are still increasing and have dense union in
$\partial(r_1\B^n)$. For every sufficiently large $\ell$, choose
$e_\ell\in E_\ell\cap B(q,\delta/2)$. Taking $y_0=e_\ell$ in the fixed-point
conclusion gives
$y_\ell\in\overline{B(q,\delta)}$ with $h_\ell(y_\ell)=e_\ell$. Set
$x_\ell=\phi(y_\ell)$. Then $F_\ell(x_\ell)=e_\ell$, while
$dh_\ell(y_\ell)=dF_\ell(x_\ell)\circ d\phi(y_\ell)$ and the invertibility of
$dh_\ell(y_\ell)$ and $d\phi(y_\ell)$ give $\Jac F_\ell(x_\ell)\ne0$. This
contradicts $F_\ell^{-1}(E_\ell)\subset\Crit(F_\ell)$ and proves the claim.

For each positive integer $t$, we next construct a discrete set
$E_t\subset\C^n\setminus t\B^n$ with a scale-$t$ growth property. Fix $t$ and
take sequences
$$
\frac t2=a_1<a_2<\cdots\longrightarrow\frac{3t}{4},
\qquad
t=r_1<r_2<\cdots\longrightarrow\infty.
$$
For every $j$, choose nonempty finite sets
$$
E_{t,j,1}\subset E_{t,j,2}\subset\cdots\subset\partial(r_j\B^n)
$$
whose union is dense in $\partial(r_j\B^n)$. Apply the claim with
$$
(a_1,a_2,r_1,r_2,c)=(a_j,a_{j+1},r_j,r_{j+1},1/t)
$$
and with the auxiliary sequence $E_\ell=E_{t,j,\ell}$. There is an index
$\ell_{t,j}$ for which the conclusion of the claim holds. Write
$E_{t,j}=E_{t,j,\ell_{t,j}}$. Consequently, every holomorphic map
$\Phi\colon a_{j+1}\B^n\to r_{j+1}\B^n$ satisfying
$$
\norm{\Phi(0)}<\frac{r_j}{2},
\qquad
\abs{\Jac \Phi(z)}\ge\frac1t
\quad\text{for some }z\in a_j\B^n
$$
also satisfies
$$
\Phi(a_j\B^n)\cap\partial(r_j\B^n)\ne\varnothing
\quad\Longrightarrow\quad
\Phi^{-1}(E_{t,j})\not\subset\Crit(\Phi).
$$
Set
$$
E_t=\bigcup_{j=1}^{\infty}E_{t,j}.
$$
Since $E_{t,j}\subset\partial(r_j\B^n)$ and $r_j\to\infty$, the set $E_t$ is
nonempty, closed, and discrete, and $E_t\subset\C^n\setminus t\B^n$.

Let $\Psi\colon t\B^n\to\C^n$ be holomorphic and suppose that
$$
\norm{\Psi(0)}<\frac t2,
\qquad
\abs{\Jac \Psi(\zeta)}\ge\frac1t
\quad\text{for some }\zeta\in\frac t2\B^n,
\qquad
\Psi^{-1}(E_t)\subset\Crit(\Psi).
$$
For example, one may take $\Psi$ to be the identity map.
We claim that
$$
\Psi\left(\frac t2\B^n\right)\subset t\B^n.
$$
The proof of \cite[Lemma~4.4]{RosayRudin1988} shows that the growth estimate
follows by downward induction once one has a starting inclusion and can pass
from the inclusion at level $k+1$ to the one at level $k$. Choose $J$ such that
$$
r_{J+1}>M\coloneqq
\max_{z\in\overline{(3t/4)\B^n}}\norm{\Psi(z)}.
$$
Since $a_{J+1}<3t/4$, this gives
$$
\Psi(a_{J+1}\B^n)\subset r_{J+1}\B^n.
$$
Fix $1\leq k\leq J$ and suppose that
$\Psi(a_{k+1}\B^n)\subset r_{k+1}\B^n$. The restriction
$\Psi|_{a_{k+1}\B^n}\colon a_{k+1}\B^n\to r_{k+1}\B^n$ is holomorphic and,
since $r_k\ge t$ and $a_k\ge t/2$, satisfies
$$
\norm{\Psi(0)}<\frac{r_k}{2},
\qquad
\abs{\Jac \Psi(\zeta)}\geq\frac1t,
\qquad
\zeta\in a_k\B^n.
$$
Our choice of $E_{t,k}$ therefore gives the stronger implication
$$
\Psi(a_k\B^n)\cap\partial(r_k\B^n)\ne\varnothing
\quad\Longrightarrow\quad
\Psi^{-1}(E_{t,k})\not\subset\Crit(\Psi).
$$
The inclusion $E_{t,k}\subset E_t$ and the condition
$\Psi^{-1}(E_t)\subset\Crit(\Psi)$ rule out
$\Psi^{-1}(E_{t,k})\not\subset\Crit(\Psi)$. Hence
$\Psi(a_k\B^n)$ avoids $\partial(r_k\B^n)$. Since this image is connected and
contains $\Psi(0)\in r_k\B^n$, it follows that
$$
\Psi(a_{k+1}\B^n)\subset r_{k+1}\B^n
\quad\Longrightarrow\quad
\Psi(a_k\B^n)\subset r_k\B^n.
$$
Starting with the inclusion at level $J+1$ and applying this implication for
$k=J,J-1,\ldots,1$ gives
$\Psi((t/2)\B^n)\subset t\B^n$, as claimed.

Set
$$
D=\bigcup_{t=1}^{\infty}E_t.
$$
For every $R>0$, the sets $E_t$ with $t\ge R$ do not meet $R\B^n$, while
each of the finitely many sets $E_t$ with $t<R$ meets $R\B^n$ in a finite
set. It follows that $D\cap R\B^n$ is finite, so $D$ is nonempty, closed, and
discrete.

Let $F\colon\C^n\to\C^n$ be nondegenerate. If
$F^{-1}(D)\subset\Crit(F)$, take $z_0\in\C^n$ such that
$\Jac F(z_0)\ne0$ and write $\eta=\abs{\Jac F(z_0)}>0$. Since $E_t\subset D$,
for every sufficiently large integer $t$,
$$
\norm{F(0)}<\frac t2,
\qquad
z_0\in\frac t2\B^n,
\qquad
\eta\ge\frac1t,
\qquad
F^{-1}(E_t)\subset\Crit(F).
$$
The scale-$t$ growth estimate with $\Psi=F|_{t\B^n}$ gives
$$
F\left(\frac t2\B^n\right)\subset t\B^n
\quad(t\gg1).
$$
The proof of \cite[Theorem~4.5]{RosayRudin1988} shows that the inclusions
$F((t/2)\B^n)\subset t\B^n$ for all sufficiently large positive integers $t$
imply that $F$ is affine. Write $F(z)=Az+b$. Nondegeneracy gives
$\det A\ne0$, so $F$ is surjective and $\Crit(F)=\varnothing$. Take $d\in D$.
Then
$$
F^{-1}(d)\ne\varnothing,
\qquad
F^{-1}(d)\cap\Crit(F)=\varnothing,
$$
contrary to $F^{-1}(D)\subset\Crit(F)$.
\end{proof}

\begin{proof}[Proof of \cref{thm:connected-blowup}]
Fix $n\ge3$, and choose $D\subset\C^n$ as in
\cref{bl:thm:regular-unavoidable}. The construction shows that $D$ meets every
compact subset of $\C^n$ in a finite set and is unbounded. Its distinct
points can therefore be enumerated as a sequence $\{a_j\}$ with
$\norm{a_j}\to\infty$. By
\cite[Proposition~2]{ForstnericGlobevnikRosay1996}, there is a proper
holomorphic embedding
$$
\gamma\colon\C\hookrightarrow\C^n
$$
whose image $A=\gamma(\C)$ contains the sequence $\{a_j\}$, hence contains $D$.
Properness makes $A$ closed; moreover, $A$ is connected, smooth, and
biholomorphic to $\C$. Its codimension is $n-1\ge2$.

Write
$$
\pi\colon Y=\Bl_A\C^n\longrightarrow\C^n
$$
for the blow-down map. If there is a
nondegenerate holomorphic map $f\colon\C^n\to Y$, choose $z_0\in\C^n$ where
$df_{z_0}$ is invertible. By the holomorphic inverse function theorem, there
is a neighborhood $U$ of $z_0$ such that $f|_U\colon U\to f(U)$ is
biholomorphic and $f(U)$ is open in $Y$.

Set $E=\pi^{-1}(A)$. By
\cite[Chapter~VII, p.~354]{DemaillyCADG}, the set $E$ is a smooth hypersurface
and $\pi\colon Y\setminus E\to\C^n\setminus A$ is biholomorphic; in particular,
$E$ has empty interior. Fix $y\in E$ and write
$s=\operatorname{codim}_{\C^n}A=n-1$. Choose local coordinates
$(z_1,\ldots,z_n)$ near $\pi(y)$ such that
$A=\{z_1=\cdots=z_s=0\}$, and a blow-up chart containing $y$, indexed by $j$,
in which $E=\{w_j=0\}$. In these coordinates,
\cite[Chapter~VII, formula~(12.2)]{DemaillyCADG} gives
$$
\pi(w_1,\ldots,w_n)
=\bigl(w_1w_j,\ldots,w_{j-1}w_j,w_j,w_{j+1}w_j,\ldots,w_sw_j,
w_{s+1},\ldots,w_n\bigr).
$$
At $y$, where $w_j=0$, this formula annihilates
$\partial/\partial w_k$ for $1\leq k\leq s$, $k\ne j$. The image of
$\partial/\partial w_j$ has $\partial/\partial z_j$-component $1$, while
$d\pi_y(\partial/\partial w_\ell)=\partial/\partial z_\ell$ for $\ell>s$.
These $n-s+1$ images are linearly independent, and hence
$$
\rank d\pi_y=n-s+1=\dim A+1=2.
$$
Since $f(U)$ is a nonempty open subset of $Y$ and $E$ has empty interior,
$f(U)\setminus E\ne\varnothing$. Choose $y_1\in f(U)\setminus E$ and let
$\zeta_1=(f|_U)^{-1}(y_1)$. The differentials $df_{\zeta_1}$ and $d\pi_{y_1}$ are
invertible, so
$$
F=\pi\circ f\colon\C^n\longrightarrow\C^n
$$
is nondegenerate. On the other hand, if $z\in F^{-1}(A)$, then $f(z)\in E$,
and the displayed rank formula gives
$$
\rank dF_z
\le \rank d\pi_{f(z)}
=2<n.
$$
Thus $F^{-1}(A)\subset\Crit(F)$ and, since $D\subset A$,
$F^{-1}(D)\subset\Crit(F)$. This contradicts
\cref{bl:thm:regular-unavoidable}. Therefore $Y$ is Brody volume hyperbolic
and, by \cite[p.~770]{Forstneric2013Survey}, is not Oka.
\end{proof}

\section{The algebraic basic Oka problem}\label{sec:abop-main}

It remains to prove \cref{thm:punctured-abop}. All cohomology groups in this
section are singular cohomology groups with rational coefficients unless
otherwise specified. We first construct a
continuous map with nonzero pullback and then compare Hodge types and weights
to show that no regular map can have the same pullback.

Consider the projective cubic
$E=\{[X:Y:Z]\in\Pj^2:Y^2Z=X^3-XZ^2\}$, denote its point at infinity by
$p=[0:1:0]$, and set $E^\circ=E\setminus\{p\}$.
On $Z=0$ the equation gives $X=0$, so $p$ is the unique point of $E$ at
infinity. For
$$
  \Phi(X,Y,Z)=Y^2Z-X^3+XZ^2
$$
one has
$$
  \Phi_X=-3X^2+Z^2,
  \qquad \Phi_Y=2YZ,
  \qquad \Phi_Z=Y^2+2XZ.
$$
On $Z=1$, a singular point would satisfy $y=0$ and $x^2=1/3$, which is
incompatible with $y^2=x^3-x$. Thus the affine curve is nonsingular, while
$\Phi_Z(p)=1$ shows that the unique point $p$ of $E$ at infinity is also
nonsingular. Hence $E$ is nonsingular. Moreover,
$x^3-x$ has the simple zeros $0,1,-1$ and is therefore not a square in
$\C(x)$. Thus $y^2-(x^3-x)$ is irreducible in $\C(x)[y]$, and hence in
$\C[x,y]$ by Gauss's lemma; its homogenization $\Phi$ is also irreducible.
Consequently, $E$ is irreducible, and
$$
  E^\circ=\{(x,y)\in\C^2:y^2=x^3-x\}.
$$
Fischer's genus formula \cite[p.~177]{Fischer2001} gives $g(E)=1$. Hence $E$
is homeomorphic to a torus. Lawson's construction
\cite[p.~209]{Lawson2003} provides a deformation retraction of $E^\circ$ onto
an embedded wedge $S^1\vee S^1$; denote its terminal retraction by
$h\colon E^\circ\to S^1\vee S^1$.

Fix an integer $n\geq2$, and set
$X_n=E^\circ\times(\Cstar)^{2n-2}$ and $Y_n=\C^n\setminus\{0\}$.
The curve $E^\circ$ is affine by its displayed equation, and
$\Cstar=\operatorname{Spec}\C[t,t^{-1}]$ is affine.
The Jacobian calculation above shows that $E^\circ$ is smooth, and each
factor $\Cstar$ is smooth. Their product $X_n$ is therefore a smooth
affine algebraic variety of complex dimension $1+(2n-2)=2n-1$. The
deformation retraction $h$ shows that $E^\circ$ is path connected, and every
$\Cstar$-factor is path connected as well. Hence $X_n$ is a smooth connected
affine algebraic variety and is an admissible source under both aBOP
conventions discussed in the Introduction. The target $Y_n$ is
a smooth connected algebraic variety, being the complement of one point in
$\A^n_\C$ with $n\geq2$.

\begin{lemma}\label{ab:lem:nonzero-pullback}
For the fixed integer $n\geq2$, there is a continuous map
$F_n\colon X_n\to Y_n$ whose pullback in degree $2n-1$ is nonzero.
\end{lemma}

\begin{proof}
Let $\lambda\colon S^1\vee S^1\to S^1$ be the map that is the
identity on the first circle and collapses the second circle to the wedge
point. Denote by $s\colon S^1\to E^\circ$ the inclusion of the first circle and set
$\rho=\lambda\circ h$. Give $S^1$ its standard orientation, and let
$\alpha\in H^1(S^1;\Q)$ be the rational cohomology orientation class normalized by
$\langle\alpha,[S^1]\rangle=1$. The identity $\rho\circ s=\id_{S^1}$ gives
$$
s^*(\rho^*\alpha)=(\rho\circ s)^*\alpha=\alpha\ne0,
$$
so $\rho^*\alpha\ne0$.

For each $\Cstar$-factor, the maps
$R_t\colon\Cstar\to\Cstar$, $R_t(z)=((1-t)+t/|z|)z$, form a deformation
retraction onto $S^1$. Its terminal retraction is
$$
r\colon\Cstar\longrightarrow S^1,
\qquad r(z)=\frac{z}{|z|}
$$
at $t=1$. Thus
$r^*:H^1(S^1;\Q)\to H^1(\Cstar;\Q)$ is an isomorphism. The maps $\rho$
and $r$ give
$$
q_n\colon X_n\longrightarrow T^{2n-1}=(S^1)^{2n-1},
\qquad
q_n(x,z_1,\ldots,z_{2n-2})
=\bigl(\rho(x),r(z_1),\ldots,r(z_{2n-2})\bigr).
$$
To construct a degree-one map from $T^{2n-1}$ to $S^{2n-1}$, give the sphere
its standard orientation and the torus its product orientation. Let
$u_n\in H^{2n-1}(S^{2n-1};\Z)$ and
$\theta_n\in H^{2n-1}(T^{2n-1};\Z)$ be the resulting cohomology orientation
classes.
Write $\pi_0\colon X_n\to E^\circ$ and $\pi_j\colon X_n\to\Cstar$,
$1\leq j\leq2n-2$, for the factor projections. Viewing $\theta_n$ with
rational coefficients, the product orientation, the K\"unneth theorem
\cite[Theorem~3.15]{Hatcher2002}, and the nonvanishing of $\rho^*\alpha$ and
$r^*\alpha$ give
$$
q_n^*\theta_n
=\pi_0^*(\rho^*\alpha)\smile\pi_1^*(r^*\alpha)\smile\cdots\smile
\pi_{2n-2}^*(r^*\alpha)\ne0.
$$

The standard product CW structure makes $T^{2n-1}$ a
$(2n-1)$-dimensional CW complex. Since $2n-1\geq3$, the Hopf classification
theorem \cite[Chapter~V, Theorems~11.6 and~11.9]{Bredon1993} gives a
continuous map
$$
c_n\colon T^{2n-1}\longrightarrow S^{2n-1}
$$
such that $c_n^*u_n=\theta_n$; equivalently, $c_n$ has degree $1$.
Identify $S^{2n-1}$ with the unit sphere in $\C^n$, and denote its inclusion
into $Y_n$ by $\iota_n$. The radial deformation retraction
$z\mapsto z/\norm{z}$ makes $\iota_n$ a homotopy equivalence. After extending
$u_n$ to rational coefficients, choose
$\eta_n\in H^{2n-1}(Y_n;\Q)$ with $\iota_n^*\eta_n=u_n$. The composite
$$
F_n=\iota_n\circ c_n\circ q_n\colon X_n\longrightarrow Y_n
$$
satisfies
$$
F_n^*\eta_n=q_n^*c_n^*\iota_n^*\eta_n=q_n^*\theta_n\ne0.
$$
Thus $F_n^*\colon H^{2n-1}(Y_n;\Q)\to H^{2n-1}(X_n;\Q)$ is nonzero.
\end{proof}

To show that no regular map is homotopic to $F_n$, we compare the Hodge types
and weights of these two cohomology groups. For every smooth complex algebraic
variety $Z$, Deligne's canonical splitting has the form
$$
H^k(Z;\C)=\bigoplus_{p,q}I^{p,q}H^k(Z).
$$
The symbol $I^{p,q}H^k(Z)$ denotes the component of Hodge bidegree $(p,q)$;
such a component has Hodge type $(p,q)$ and weight $p+q$. The weight filtration
on the cohomology group is the increasing filtration given by
$$
W_\ell H^k(Z;\C)=\bigoplus_{p+q\leq\ell}I^{p,q}H^k(Z),
\qquad
W_\ell H^k(Z;\Q)=H^k(Z;\Q)\cap W_\ell H^k(Z;\C)
$$
\cite[Lemma~1.2.11]{Deligne1971}. Thus $W_\ell H^k(Z;\Q)$ consists precisely
of the rational cohomology classes whose complex Hodge components all have
weight at most $\ell$. For smooth $Z$, the possible weights in
$H^k(Z;\Q)$ range from $k$ to $2k$ \cite[Corollary~3.2.15]{Deligne1971}, so
$$
0=W_{k-1}H^k(Z;\Q)\subseteq W_kH^k(Z;\Q)\subseteq\cdots\subseteq
W_{2k}H^k(Z;\Q)=H^k(Z;\Q).
$$
The group $H^k(Z;\Q)$ is pure of weight $w$ when
$I^{p,q}H^k(Z)=0$ for every $p+q\ne w$; equivalently,
$W_{w-1}H^k(Z;\Q)=0$ and
$W_wH^k(Z;\Q)=H^k(Z;\Q)$. In this pure case one writes
$H^{p,q}(Z)=I^{p,q}H^k(Z)$, and
$$
H^k(Z;\C)=\bigoplus_{p+q=w}H^{p,q}(Z)
$$
is the usual Hodge decomposition; thus $H^{p,q}(Z)$ denotes its component of
Hodge type $(p,q)$. In particular, this is the standard notation for the
Hodge decomposition of a smooth projective variety.

Following the rational Tate convention associated with
\cite[Definition~2.1.13]{Deligne1971}, $\Q(-1)$ denotes the one-dimensional
rational Hodge structure whose complexification has Hodge type $(1,1)$ and
weight $2$. For a rational Hodge structure $V$,
its Tate twist is $V(-1)=V\otimes_\Q\Q(-1)$; it shifts every Hodge type
$(p,q)$ to $(p+1,q+1)$ and increases the weight by $2$. For $m\geq1$, write
$\Q(-m)=\Q(-1)^{\otimes m}$; twisting by $(-m)$ shifts each Hodge type
$(p,q)$ to $(p+m,q+m)$ and increases the weight by $2m$.

\begin{proposition}\label{ab:prop:weights}
For the fixed integer $n\geq2$, the group $H^{2n-1}(Y_n;\Q)$ is
one-dimensional of Hodge type $(n,n)$ and is pure of weight $2n$.
Moreover, $H^{2n-1}(X_n;\C)$ is the direct sum of two one-dimensional Hodge
components of types
$(2n-1,2n-2)$ and $(2n-2,2n-1)$, and $H^{2n-1}(X_n;\Q)$ is pure of
weight $4n-3$.
\end{proposition}

\begin{proof}
The radial deformation retraction of $Y_n$ onto $S^{2n-1}$ gives
$$
H^j(Y_n;\Z)\cong
\begin{cases}
\Z, & j=0,2n-1,\\
0, & \text{otherwise}.
\end{cases}
$$
Since these integral cohomology groups are free, the universal coefficient
theorem \cite[Theorem~3.2]{Hatcher2002} gives
$$
H^j(Y_n;\Q)\cong H^j(Y_n;\Z)\otimes_\Z\Q\cong
\begin{cases}
\Q, & j=0,2n-1,\\
0, & \text{otherwise}.
\end{cases}
$$
The Gysin exact sequence for the smooth closed subvariety
$\{0\}\subset\A^n_\C$ of codimension $n$
\cite[Proposition~2.134]{BurgosFresan2020} contains the exact sequence of
mixed Hodge structures
$$
0=H^{2n-1}(\A^n_\C;\Q)\longrightarrow H^{2n-1}(Y_n;\Q)
\longrightarrow H^0(\{0\};\Q)(-n)
\longrightarrow H^{2n}(\A^n_\C;\Q)=0.
$$
Consequently, there is an isomorphism of mixed Hodge structures
$H^{2n-1}(Y_n;\Q)\cong\Q(-n)$. Thus the entire group $H^{2n-1}(Y_n;\Q)$
has Hodge type $(n,n)$ and is pure of weight $2n$.

Next compute the Hodge types and weights contributed by the factors of $X_n$.
The deformation retraction $h$ cited above gives
$H^j(E^\circ;\Q)=0$ for $j>1$. Applied to $E^\circ=E\setminus\{p\}$,
\cite[Example~2.136]{BurgosFresan2020} gives an isomorphism of mixed Hodge structures
$H^1(E^\circ;\Q)\cong H^1(E;\Q)$.
By \cite[Example~1.5.17]{CattaniEtAl2014},
$\dim_\C H^{1,0}(E)=\dim_\C H^{0,1}(E)=g(E)=1$. The isomorphism above shows that
$H^1(E^\circ;\C)$ has two one-dimensional Hodge components of types
$(1,0)$ and $(0,1)$, and $H^1(E^\circ;\Q)$ is pure of weight $1$.

The radial deformation retraction of $\Cstar$ onto $S^1$ gives
$$
H^j(\Cstar;\Q)\cong
\begin{cases}
\Q, & j=0,1,\\
0, & j>1.
\end{cases}
$$
Applied to $\Cstar=\Pj^1\setminus\{0,\infty\}$,
\cite[Example~2.136]{BurgosFresan2020} gives an isomorphism of mixed Hodge
structures $H^1(\Cstar;\Q)\cong\Q(-1)$.
Hence this group has Hodge type $(1,1)$ and is pure of weight $2$.

Since each of the $2n-1$ factors of $X_n$ has no cohomology above degree $1$,
Deligne's K\"unneth theorem \cite[Proposition~8.2.10]{Deligne1974} directly
gives the following isomorphism of mixed Hodge structures:
$$
  H^{2n-1}(X_n;\Q)
  \cong
  H^1(E^\circ;\Q)
  \otimes H^1(\Cstar;\Q)^{\otimes(2n-2)}.
$$
Under this isomorphism, Hodge bidegrees add across the tensor factors, so the
two summands have types
$$
(1,0)+(2n-2)(1,1)=(2n-1,2n-2),
\qquad
(0,1)+(2n-2)(1,1)=(2n-2,2n-1).
$$
Both summands are one-dimensional, so the entire group
$H^{2n-1}(X_n;\Q)$ is pure of weight $4n-3$. Since $2n<4n-3$,
$$
W_{2n}H^{2n-1}(X_n;\Q)=0.
$$
\end{proof}

\begin{proof}[Proof of Theorem~\ref{thm:punctured-abop}]
Fix $n\geq2$. The proof of Proposition~\ref{ab:prop:weights} gives
$$
W_{2n}H^{2n-1}(Y_n;\Q)=H^{2n-1}(Y_n;\Q),
\qquad
W_{2n}H^{2n-1}(X_n;\Q)=0.
$$

By \cite[Theorem~3.2.5(iii)]{Deligne1971}, every regular map
$g\colon X_n\to Y_n$ satisfies
$$
  g^*\bigl(H^{2n-1}(Y_n;\Q)\bigr)
  =g^*\bigl(W_{2n}H^{2n-1}(Y_n;\Q)\bigr)
  \subseteq W_{2n}H^{2n-1}(X_n;\Q)
  =0.
$$
Thus every regular map $X_n\to Y_n$ induces the zero map in degree
$2n-1$.

By \cref{ab:lem:nonzero-pullback}, there is a continuous map
$F_n\colon X_n\to Y_n$ with $F_n^*\neq0$. If $Y_n$ satisfied aBOP, then,
since $X_n$ is a smooth connected affine algebraic variety, the definition
of aBOP would provide a regular map $g\colon X_n\to Y_n$ homotopic to $F_n$.
Homotopy invariance of singular cohomology would then give
$$
  F_n^*=g^*:H^{2n-1}(Y_n;\Q)
  \longrightarrow H^{2n-1}(X_n;\Q),
$$
which is impossible because $F_n^*\neq0$ and $g^*=0$. Therefore
$Y_n$ does not satisfy aBOP. Since $n\geq2$ is arbitrary, the theorem
follows.
\end{proof}

\section*{Acknowledgments}

The author thanks his advisor, Professor Song--Yan Xie, for his guidance and
support.

Generative artificial intelligence tools were used in the preparation and
revision of this manuscript. The author verified all details and takes full
responsibility for the content.

\end{document}